\documentclass[12pt,a4paper,twoside]{article}

\usepackage{amsmath,amssymb,amsfonts,amsthm}
\usepackage{times,hyperref,color}
\usepackage{graphicx}

\let\mathcal\mathscr
\newtheorem{theorem}{\bf Theorem}[section]

\newtheorem{lemma}[theorem]{\bf Lemma}

\newtheorem{conjecture}{\bf Conjecture}

\usepackage{epsfig}
\begin{document}

\title{On a Conjecture of Erd\H os on Size Ramsey Number of Star Forests}
\author{A. Davoodi$^1$\footnote{The research of the first author was supported by the Czech Science Foundation, grant number GA19-08740S.}
	, R. Javadi$^{2,3}$\footnote{The research of the second author was in part supported by a grant from IPM (No. 1400050420).}, A. Kamranian$^{4}$, G. Raeisi$^{4}$\vspace{0.4cm}\\
{\footnotesize  $^{1}$The Czech Academy of Sciences, Institute of Computer Science, Pod Vod\'{a}renskou v\v{e}\v{z}\'{\i} 2, 182 07 Prague, Czech Republic.}\\
 {\footnotesize $^{2}$Department of Mathematical Sciences, Isfahan University of Technology, 	84156-83111, Isfahan, Iran}\\
 {\footnotesize $^{3}$School of Mathematics, Institute for Research in Fundamental Sciences (IPM), P.O. Box: 19395-5746, Tehran, Iran} \\
{\footnotesize $^{4}$Department of Mathematical Sciences, Shahrekord University, P.O.Box 115, Shahrekord, Iran}\vspace{0.4cm}\\
{\footnotesize {\bf E-mails:} davoodi@cs.cas.cz, rjavadi@iut.ac.ir, azamkamranian@stu.sku.ac.ir, g.raeisi@sci.sku.ac.ir}}

\date{}
\maketitle

\begin{abstract}
{\footnotesize
Given graphs $ F_1, F_2$ and $G$, we say that $G$ is Ramsey for $(F_1,F_2)$ and we write $G\rightarrow(F_1, F_2)$, if for every edge coloring of $G$ by red and blue, there is either a red copy of $F_1$ or a blue copy of $F_2$ in $G$. The size Ramsey number $\hat{r}(F_1, F_2)$ is defined as the minimum number of edges of a graph $G$ such that $G\rightarrow(F_1, F_2)$. This paper provides the exact value of $\hat{r}(F_1, F_2)$ for many pairs $(F_1, F_2)$ of star forests, giving a partial solution to a conjecture of Burr et al. (Ramsey-minimal graphs for multiple copies, Indagationes Mathematicae, 81(2) (1978), 187-195). }
\end{abstract}

\bigskip
\section{Introduction}
In this paper, we are only concerned with undirected simple finite graphs without isolated vertices and we follow \cite{Boundy} for terminology and notations not defined here.  For a given graph $G$, we denote its vertex set,  edge set,  maximum degree, minimum degree and edge chromatic number (chromatic index) of $G$ by $V(G)$, $E(G)$,  $\Delta(G)$, $\delta(G)$ and $\chi'(G)$, respectively.
The {\it order} of a graph $G$ is the cardinality of its vertex set, and the {\it size} of $G$, $e(G)$, is the cardinality of its edge set. For a vertex $v\in V(G)$, we use $\deg{(v)}$ and $N(v)$ to denote the degree and the set of neighborhoods of $v$ in $G$, respectively.  Also, $ G-v $ stands for the graph obtained from $ G $ by deleting $ v $ and all leave neighbors of $ v $. The {\it star} graph on $n+1$ vertices is denoted by $K_{1,n}$, consisting of one root vertex connected to $n$ leaves. Also, by a {\it matching} of size $m$, $mK_2$, we mean $m$ independent edges. Also, for given graphs $G$ and $H$, we use $G\sqcup H$ to denote the disjoint union of $G$ and $H$ and we use $mG$ to denote the disjoint union of $m$ copies of $G$.

\medskip
A {\it factor} in a graph $G$ is a spanning subgraph of $G$ and a {\it $k$-factor} is a $k$-regular factor in $G$. In particular,  1-factors are usually called perfect matchings. A graph $G$ is called {\it $k$-factorable} if the edges of $G$ can be decomposed into $k$-factors. It is well known that graphs of even degree can be decomposed into edge-disjoint cycles \cite{Boundy}. For regular graphs of even degree, the cycles in some decomposition can be grouped to form 2-factors, as presented in the following theorem.

\begin{theorem}[Petersen \cite{5}]\label{peterthm}
Every regular graph of even degree is 2-factorable.
\end{theorem}

\medskip
For given graphs $F_1, F_2$ and $G$, we write $G\rightarrow(F_1, F_2)$ and we say $G$ is {\it Ramsey graph} for the pair $(F_1,F_2)$ if in any red-blue coloring of the edges of $G$, there is either a red copy of $F_1$ or a blue copy of $F_2$.  A $2$-edge coloring of $G$  is called  $(F_1, F_2)$-{\it free coloring} if $G$ avoid monochromatic $F_i$ in the $i$-th color class, for each $i\in\{1,2\}$. For given graphs $F_1, F_2$, the {\it size Ramsey number} $\hat{r}(F_1, F_2)$ is defined as the  minimum number of the edges of a graph $G$ such that $G\rightarrow(F_1, F_2)$.  The investigation of size Ramsey numbers was first initiated by Erd{\H{o}}s et al. \cite{14} in 1978. For results concerning size Ramsey numbers, we refer the reader to \cite{14, surveysize, jo, pkh} and references therein. In this paper, we study the size Ramsey number of star forests (a star forest, sometimes also referred as a galaxy, is a forest whose each component  is a star) and the exact value of $\hat{r}(F_1, F_2)$ is determined for some pairs $(F_1, F_2)$  of star forests.


\medskip
In 1978 \cite{7}, Burr  et al. proposed the following conjecture, determining the exact value of the 2-color size Ramsey number of star forests.

\begin{conjecture}[Burr, Erd{\H o}s, Faudree, Rousseau and Schelp \cite{7}]\label{conj}
For given positive integers $n_1\geq n_2\geq \cdots\geq n_s$ and $m_1\geq m_2\geq \cdots\geq m_t$, let $F_1=\sqcup_{i=1}^{s}K_{1, n_i}$ and $F_2=\sqcup_{j=1}^{t}K_{1, m_j}$. Then, $\hat{r}(F_1, F_2)= \sum_{k=2}^{s+t}l_k,$
where, $l_k=\max\{n_i+m_j-1 : i+j=k\}$.
\end{conjecture}

This conjecture was verified in \cite{7} for star forests containing stars of the same sizes.

\begin{theorem}[Burr, Erd{\H{o}}s, Faudree, Rousseau and Schelp \cite{7}]\label{samesizestars}
For positive integers $s$, $t$, $m$ and $n$,
$\hat{r}(sK_{1,n}, tK_{1,m})=(s+t-1)(m+n-1).$ Moreover, if $ G\rightarrow(sK_{1,n},tK_{1,m}) $ and has $ (s+t-1)(m+n-1) $ edges, then $ G=(s+t-1)K_{n+m-1} $ or $ n=m=2 $ and $ G=lK_3\sqcup (s+t-l-1)K_{1,3}$, for some $l$, $1\leq l\leq s+t-1$.
\end{theorem}

\medskip
Subsequently, the restriction that each star forest could only have stars of the same size was relaxed and a more general condition was considered by Gy{\H{o}}ri and Schelp in \cite{24}. Although they did not completely verify the conjecture for all star forests, they did so for a large class of forests, which provides substantial support for Conjecture \ref{conj}.

\begin{theorem}[Gy{\H{o}}ri and Schelp \cite{24}]
For positive integers $n_1\geq n_2\geq \cdots\geq n_s$ and $m_1\geq m_2\geq \cdots\geq m_t$, let $F_1=\sqcup_{i=1}^{s}K_{1, n_i}$ and $F_2=\sqcup_{j=1}^{t}K_{1, m_j}$. For $k=2,\ldots,s+t$, set $l_k=\max\{n_i+m_j-1 : i+j=k\}$. If ${l_k \choose 2}>\sum_{i=k}^{s+t}l_i$, for all $2\leq k\leq s+t$, then $\hat{r}(F_1, F_2)=\sum_{k=2}^{s+t}l_k$.
\end{theorem}

\bigskip
In Theorem~\ref{samesizestars}, some cases are missed for Ramsey minimal graphs. So, in this paper, we first present a much shorter proof for Theorem \ref{samesizestars} including all equality cases (see Theorem~\ref{starsss}). In addition, we prove Conjecture~\ref{conj} for many pairs $(F_1=\sqcup_{i=1}^{s}K_{1, n_i},F_2=\sqcup_{j=1}^{t}K_{1, m_j})$ of star forests, including the cases $s=1$ (Theorem~\ref{star-forest}), $s=2$ and $n_1=n_2$ (Theorem~\ref{thm:s=2}), all $n_i$'s and $m_i$'s are odd (Theorem~\ref{odd-starforest}) and all $n_i$'s are equal to an odd number and $m_1$ is also odd (Theorem~\ref{odd-stars-forest}). Moreover, for some pairs $(F_1, F_2)$ of star forests, Ramsey minimal graphs, i.e.  graphs $F$ with $F\rightarrow(F_1,F_2)$ and $e(F)=\hat{r}(F_1,F_2)$  will be classified exactly.

\section{Main Results}
In this section, the main results of the paper will be presented. Note that by the notations of Conjecture \ref{conj} and  using the simple fact that $K_{1,n+m-1}\rightarrow(K_{1,n},K_{1,m})$, we deduce that if $F=\sqcup_{k=2}^{s+t}K_{1,l_k}$, then $F\rightarrow(F_1,F_2)$ and so, $\hat{r}(F_1,F_2)$ is upper bounded by $e(F)=\sum_{k=2}^{s+t}l_k$. This is certainly the case for all results of this section, therefore we shall always prove just the claimed lower bound for the size Ramsey number of star forests. Thus, to determine $\hat{r}(F_1,F_2)$ for a given pair  $(F_1,F_2)$ of star forests, it is sufficient to show that if $F$ is a graph such that $F\rightarrow(F_1,F_2)$, then $e(F)\geq \sum_{k=2}^{s+t}l_k$.

\bigskip
In what follows, we present a much shorter proof for Theorem \ref{samesizestars}. Beforehand, we prove the following simple lemma that will be used in the proof of next results.

\begin{lemma}\label{decompose}
Let $n, m$ be positive integers and let $G$ be a graph such that either $\Delta(G)\leq m+n-3$, or $\Delta(G)\leq m+n-2$ and $m,n$ are both odd.  Then,  $G$ has a $(K_{1,n},K_{1,m})$-free coloring.
\end{lemma}
\begin{proof}
First, suppose that $\Delta(G)\leq m+n-3$. Then, by Vizing's Theorem, there is a proper edge-coloring $c$ for $G$ with at most $m+n-2$ colors. Partition the color classes in $c$ into two sets $A_1,A_2$ of sizes at most $n-1,m-1$, respectively. For each $i\in\{1,2\}$, let $G_i$ be the subgraph of $G$ induced by the edges of colors in $A_i$. Clearly, $G_1$ and $G_2$  decompose the edge set of $G$ and $\Delta(G_1)\leq n-1$ and $\Delta(G_2)\leq m-1$. Coloring all edges of $G_i$, $i\in\{1,2\}$, by the $i$-th color yields a $(K_{1,n},K_{1,m})$-free coloring for $G$, as desired.

Now, suppose that $\Delta(G)=m+n-2$ and $m,n$ are both odd. 
Then,  $G$ can be embedded into a $ \Delta(G)$-regular graph $ H $. By Theorem~\ref{peterthm}, $H$ can be decomposed into $(m+n-2)/2$ two-factors. Now, color all the edges of $G$ which are in the first $(n-1)/2$ two-factors of $H$ by red and other edges of $G$ by blue.  Clearly, the maximum degree of the red (resp. blue) subgraph of $G$ is at most $n-1$ (resp. $m-1$). Therefore, this is a $(K_{1,n},K_{1,m})$-free coloring for $G$. This completes the proof.
\end{proof}
Note that the above lemma is not the case when $\Delta(G)=m+n-2$ and $n$ or $m$ is even. For instance, when $m=2$, an $n$-regular graph without a perfect matching would be a counterexample. Now, we are ready to give an alternative proof for Theorem~\ref{samesizestars} including missing extremal cases.
\begin{theorem}\label{starsss}
For positive integers $s$, $t$, $m$ and $n$, $n\geq m$, we have
$\hat{r}(sK_{1,n}, tK_{1,m})=(s+t-1)(m+n-1).$ Moreover, if $ G\rightarrow(sK_{1,n},tK_{1,m}) $ and $ e(G)=(s+t-1)(m+n-1) $, then $ G=(s+t-1)K_{n+m-1} $ or $ n=m=2 $ and $ G=lK_3\sqcup (s+t-l-1)K_{1,3} $, or $s=m=1$ and $n=2$ and $ G=lC_4\sqcup (t-2l)K_{1,2} $, for some nonnegative integer $l$.
\end{theorem}
\begin{proof}
Since $ K_{1,m+n-1}\rightarrow(K_{1,n},K_{1,m}) $, we have $ (s+t-1)K_{1,m+n-1}\rightarrow(sK_{1,n},tK_{1,m}) $.
Thus, $\hat{r}(F_1,F_2)\leq(s+t-1)(m+n-1)$. Now, we prove the lower bound by induction on $s+t$. Let $ G $ be a graph such that $ G\rightarrow(sK_{1,n},tK_{1,m}) $.
If $ \Delta(G)\leq m+n-3 $, then by Lemma \ref{decompose}, there is a $(K_{1,n},K_{1,m})$-free coloring of $G$, a contradiction. Thus, $\Delta(G)\geq m+n-2$.
Let $ v $ be a vertex of maximum degree in $ G $.

\bigskip
\noindent{\bf Claim.} $(G-v) \rightarrow((s-1)K_{1,n},tK_{1,m})$.

\medskip
To prove the claim, 
on the contrary, suppose that $(G-v) \not\rightarrow((s-1)K_{1,n},tK_{1,m})$. Then, by the definition, there is a  $((s-1)K_{1,n},tK_{1,m})$-free coloring  of $G-v$. This coloring can be extended to a $(sK_{1,n},tK_{1,m})$-free coloring of $G$ by coloring all edges incident  with $v$ by the first color,  contradicting that $G\rightarrow(sK_{1,n},tK_{1,m})$. This observation completes the proof of the claim.

\medskip
Thus, $ G-v\rightarrow((s-1)K_{1,n},tK_{1,m}) $ and by the induction hypothesis, $ e(G-v)\geq (s+t-2)(m+n-1)$.
Now, if $\Delta(G)\geq m+n-1$, then  $e(G)\geq (s+t-1)(m+n-1) $ and we are done.

Thus, assume that $ \Delta(G)=m+n-2 $ and there are $ s+t-1 $ vertices, say $ v_1,\ldots, v_{s+t-1} $, in $ G $ such that $ d_{G_{i-1}}(v_i)=m+n-2 $, in which $ G_0=G $ and $ G_i=G-\{v_1,\ldots, v_{i}\} $. Therefore, $ v_i $'s form an independent set in $ G $.

Let $ W=\{v_1,\ldots, v_{s+t-1}\} $ and let $ B $ be the bipatrtite graph induced by the edges between $ W $ and $ V(G)\backslash W $ in $ G $. By Vizing's theorem and the fact that bipartite graphs are in class $ I $, $ \chi'(B)\leq\Delta(B)=\Delta(G)=m+n-2 $. Color all the first $ n-1 $ color classes by red and the $ m-1 $ remaining classes by blue. Clearly, this is a $ (K_{1,n}, K_{1,m}) $-free coloring of $ B $. We extend this coloring to a coloring for $ G $ by giving a $ 2$-coloring for (at most) $ s+t-2 $ remaining edges. Obviously, these edges lie in $ V(G)\backslash W $. Select arbitrarily $ s-1 $ of them and color these edges by red. Finally, color the remaining edges by blue. Since $ B $ has no red copy of $ K_{1,n} $ and each red edge in $ V(G)\backslash W $ participates in at most one vertex disjoint red copy of $ K_{1,n} $, there is no red copy of $ sK_{1,n} $ in $ G $. Similarly, there is no blue $ tK_{1,m} $ in $ G $. But this contradicts the assumption $ G\rightarrow(sK_{1,n},tK_{1,m}). $ Hence, $\hat{r}(F_1,F_2)=(s+t-1)(m+n-1)$.

\bigskip
Now, we are going to characterize the extremal structures. For simplicity, define the graph $H_{m,n}$ to be either $K_{1,3}$ or $K_3$ if $m=n=2$ and to be $K_{1,m+n-1}$, otherwise. Now, let $ G\rightarrow(F_1,F_2) $ and $ e(G)=(s+t-1)(n+m-1) $. There are two possibilities for $ G $. The first is $ \Delta(G)=m+n-1 $ and the second is $ \Delta(G)=m+n-2 $ and $ e(G_{s+t-1})=s+t-1 $.
\medskip

For the first case, $e(G-v)=(s+t-2)(m+n-1)$ and $G-v\rightarrow ((s-1)K_{1,n}, tK_{1,m})$ and $ G-v\rightarrow (sK_{1,n}, (t-1)K_{1,m}) $. Without loss of generality, we may assume that all isolated vertices of $ G-v $ are removed.

First, suppose $s=m=1$ and $n=2$. By the induction hypothesis, $G-v $ is the disjoint union of some $C_4$'s and $K_{1,2}$'s. Note that $\Delta(G)=m+n-1=2$ and so $v$ is not adjacent to vertices of $C_4$'s and the centers of $K_{1,2}$'s. If $N(v)$ is disjoint from $V(G-v)$ or $v$ is adjacent to both leaves of one $K_{1,2}$, then $G$ is also union of some $C_4$'s and $K_{1,2}$'s and we are done. Otherwise, $ G $ has a connected component isomorphic to either a path $ P_5 $ or $ P_7 $.
It is easy to see that $P_5\not\rightarrow (K_{1,2},2 K_{1,1})$ and $P_7\not\rightarrow (K_{1,2},3 K_{1,1})$ (for this, we can color first and last edges of the path by red and others by blue). Hence, $G\not\rightarrow (K_{1,2},t K_{1,1})$.

Now, suppose that either $s\neq 1$, or $m\neq 1$, or $n\neq 2$. Then, by the induction hypothesis,  $G-v $ is the disjoint union of $s+t-2$ graphs $H_{m,n} $. Again, if $N(v)$ is disjoint from $V(G-v)$, then $G$ is the disjoint union of $s+t-1$ graphs $H_{m,n}$ and we are done.
Now, suppose that $N(v)$ contains a vertex $ u $ in a copy of $H_{m,n}$, say $ H_0 $, in $G-v$. If  $ H_0 $ is a star, let $ w $ be the center of $ H_0 $ and if $ H_0 $ is a triangle then let $ w\neq u $ be an arbitrary vertex of $ H_0 $.

Denote the edges incident with $v$ by $ H_1 $ and color edges of $H_0$ and $ H_1 $ such that
in each $ H_i $, $ i=1,2 $, there are exactly $n$ red edges and $m-1$ blue edges, the edges $wu$ and $uv$ are both red and there is no blue copy of $K_{1,m}$ in $ H_0\sqcup H_1 $ (this can be done because $v$ and $w$ cannot be the center of a blue $K_{1,m}$ and in case $m=2$, since $n\geq m$, we can avoid a blue $K_{1,2}$). Also, since $wu$ and $uv$ are both red, there is no two vertex disjoint red copies of $K_{1,n}$ in $ H_0\sqcup H_1 $.
Now, for the remaining edges, color $s-2$ copies of $H_{m,n}$ by red and $t-1$ copies of $H_{m,n}$ by blue. So, there is at most $s-1$ disjoint red copies of $K_{1,n} $ and $t-1$ disjoint blue copies of $K_{1,m}$. Thus, $G\not \rightarrow (sK_{1,n},tK_{1,m})$.

\medskip
For the second case, i.e. when $\Delta(G)=m+n-2$ and $e(G_{s+t-1})=s+t-1$,
if $ m=1 $, then $ \Delta(G)=n-1 $ and we color all the edges by red. Therefore, we may suppose that $ m,n\geq 2 $.
If $ m=n=2 $, then $ \Delta(G)=2 $ and $ G $ is disjoint union of some paths and cycles. We color edges of each path and cycle alternatively by blue and red. Let $ \lambda $ be the number of odd cycles in $ G $. Clearly, the number of monochromatic copies of $ K_{1,2} $ is equal to $ \lambda $. Since $ e(G)=3(s+t-1) $, we have $ \lambda\leq s+t-1 $. If $ \lambda\leq s+t-2 $, then we can color $ G $ such that there are at most  $s-1$ red copies of $K_{1,2}$ and at most $t-1$ blue copies of $K_{1,2}$, and thus, $ G\not\rightarrow(F_1,F_2) $. If $ \lambda=s+t-1 $, then $ G=(s+t-1)K_3 $ and we are done.

Now, let $ m\geq2 $, $ n\geq 3 $  and we show that $ G\not\rightarrow(F_1,F_2) $. Recall that the edge set of $ B $ can be colored by red and blue so that $ B\not\rightarrow(K_{1,n},K_{1,m}) $.
 Consider the induced subgraph $F$ on $ V(G)\backslash W $.  By the assumption, $F$ has exactly $ s+t-1 $ edges. If two of these edges, say $ e_1 $ and $ e_2 $, are adjacent, then we color $ e_1, e_2 $ and $ s-2 $ more arbitrarily selected edges of $F$ by red and $ t-1 $ remaining edges of $F$ by blue. Clearly, there are at most $t-1$ disjoint blue copies of $K_{1,m}$ and at most $ (s-1)$ disjoint red copies of $K_{1,n} $ in $ G $ which is a contradiction.
Hence, the edges $ e_1, e_2, \ldots, e_{s+t-1} $ form a matching in $F $.
Note that every red $ K_{1,n} $ in $F_1$ or blue $ K_{1,m} $ in $F_2$ should contain one of $ e_i $'s.
Now, we claim that for all $ i $, $ e_i $ has one endpoint, say $ u_i $, which is incident with $ n-1 $ red edges in $ B $ and another endpoint, say $ v_i $, which is incident with $ m-1 $ blue edges in $ B $.
To see this, note that if both endpoints of $e_i$ are incident with at most $n-2$ red edges in $B$, then we can color $e_i$ and $s-1$ more edges in $F$ by red and $t-1$ remaining edges by blue and this implies that $ G\not\rightarrow(F_1,F_2) $.
Thus, there is an endpoint of $e_i$, say $u_i$, with red degree equal to $n-1$ in $B$. Similarly, there is an endpoint of $e_i$, say $v_i$, with blue degree equal to $m-1$ in $B$. Also, since $\Delta(B)= m+n-2$, $u_i$ and $v_i$ are distinct.
This proves the claim.

On the other hand, since $ B $ has exactly $ (s+t-1)(m+n-2) $ edges, we have $ \deg_B(u_i)=n-1 $ and $ \deg_B(v_i)=m-1 $. Recall that we applied Vizing's theorem to color the edges of $ B $. Now, we exchange color of the first and the last classes. In this way, every $ u_i $ is incident with $ n-2 $ red edges and one blue edge, and every $ v_i $ is incident with $ m-2 $ blue edges and one red edge.
Now, we color all the $ s+t-1 $ edges in $ F $ by red. Being $n\geq 3$ guarantees that neither a red $ K_{1,n} $ nor a blue $ K_{1,m} $ is seen. Hence, $ G\not\rightarrow(sK_{1,n}, tK_{1,m}) $.
\end{proof}

\bigskip
Now, for the next step, we are going to prove Conjecture~\ref{conj} when $s=1$. In other words, we determine the size Ramsey number of a star versus an arbitrary star forest. Moreover, Ramsey minimal graphs will be completely characterized.

\begin{theorem}\label{star-forest}
For given positive integers $n$ and $m_1\geq m_2\geq \cdots\geq m_t\geq2$, we have $\hat{r}(K_{1,n}, \sqcup_{j=1}^{t}K_{1, m_j}) = \sum_{j=1}^{t}(n+m_j-1)$. Moreover, if $F\rightarrow(K_{1,n},\sqcup_{j=1}^{t}K_{1, m_j})$ and $e(F)=\sum_{j=1}^{t} (n+m_j-1)$, then $F={\bigsqcup}_{j=1}^{t}G_j$, where for each $ j $, either $ G_j=K_{1,n+m_j-1} $, or $ n=m_j=2 $ and $G_j=K_3$.
\end{theorem}
\begin{proof}
	We use induction on $n$ to show that if $F\rightarrow(K_{1,n}, \sqcup_{j=1}^{t}K_{1, m_j})$, then $e(F)\geq \sum_{j=1}^{t}(n+m_j-1)$. If $n=1$, then $F\rightarrow(K_2, \sqcup_{j=1}^{t}K_{1, m_j})$ implies that $e(F)\geq\sum_{j=1}^{t}m_j$.
	
	\medskip
	Now, let $n\geq 2$ and  $F$ be a graph such that $F\rightarrow(K_{1,n},\sqcup_{j=1}^{t}K_{1, m_j})$. Also, let $M$ be the maximum matching in $F$. Clearly $|M|\geq t$, because if all edges of $F$ are colored by blue, then $\sqcup_{j=1}^{t}K_{1, m_j}\subseteq F$. Let $F\setminus M$ be the graph obtained from $F$ by deleting all edges of $M$.
	
	\bigskip
	\noindent{\bf Claim.} $(F\setminus M) \rightarrow(K_{1,n-1}, \sqcup_{j=1}^{t}K_{1, m_j})$.
	
	\medskip
	To prove the claim, consider a red/blue coloring of $ F\setminus M $ and extend it to a coloring of $ F $ by coloring the edges of $ M $ by red. Hence, there is either a red $ K_{1,n} $ or a blue $ \sqcup_{j=1}^{t}K_{1, m_j} $ in $ F $.
	In the latter case, $ F\setminus M $ contains a blue $ \sqcup_{j=1}^{t}K_{1, m_j} $ and in the earlier case, at most one edge of the red $ K_{1,n} $ is in $ M $, so $ F\setminus M $ contains a red $ K_{1,n-1} $. This completes the proof of the claim.
	
	\bigskip
Applying the induction hypothesis and the claim, we have 
	$$e(F\setminus M)\geq\sum_{j=1}^{t}(m_j+n-2).$$
	Therefore,  $e(F)= e(M)+e(F\setminus M)\geq\sum_{j=1}^{t}(n+m_j-1)$ and we are done.
	
	\bigskip
	Now, we use induction on $n$ to  characterize the extremal structures. Suppose that $F\rightarrow (K_{1,n},\sqcup_{j=1}^{t}K_{1, m_j})$ and $e(F)= \sum_{j=1}^t n+m_j-1 $. For $n=1$, it is obvious that $F=\bigsqcup_{j=1}^{t}K_{1,m_j}$. Now, let $n\geq 2$ and $M$ be the maximum matching of $F$. By the above arguments, $e(M)\geq t$ and $(F\setminus M) \rightarrow(K_{1,n-1}, \sqcup_{j=1}^{t}K_{1, m_j})$. Thus, $ e(F\setminus M)=\sum_{j=1}^{t}(n+m_j-2) $ and $ e(M)=t $.
	
	\medskip
By the induction hypothesis, $F\setminus M$ contains exactly $t$ components such that each component is either a star or a triangle.
We claim that there is no triangle in $ F\setminus M $. Otherwise, we have $ n=3 $ and in this case, color all the triangles and two arbitrary edges of each star in $F\setminus M$ by red and all the remaining edges of $ F $ by blue. Clearly, there is no red copy of $ K_{1,n} $ and the number of blue edges is strictly less than $ \sum_{j=1}^{t} m_j $ which is a contradiction. Hence, $ F\setminus M=\sqcup_{j=1}^t K_{1, n+m_j-2}$.

Now, we claim that for every edge $ e\in M $, the endpoints of $ e $ cannot be in two different components of $ F\setminus M $. On the contrary, suppose that the endpoints of $ e $ are in two components of $ F\setminus M $ say $ S_1 $ and $ S_2 $. Color $ n-1 $ edge of each star in $ F\setminus M $ by red and all other edges of $ F $ by blue such that $ S_1\cup S_2\cup e $ contains a blue copy of $ P_4 $, a path of length three. It is obvious that there is no red copy of $ K_{1,n} $ and the number of blue edges is exactly $ \sum_{j=1}^{t} m_j $.  Existence of a blue $ P_4 $ in $ F $ implies that $ F $ contains no blue copy of $ F_2 $. This observation proves the claim.

The maximality of $ M $ implies that each edge $ e $ of $M$ is either incident with a center of a star in $F\setminus M$, or $n=3$ and $e$ is incident with the leaves of a $K_{1,2}$ in $ F\setminus M $. Therefore, $F={\bigsqcup}_{j=1}^{t}G_j$, where for each $ j $, either $ G_j=K_{1,n+m_j-1} $, or  $ n=m_j=2 $ and $G_j= K_3 $. This completes the proof.
\end{proof}

\bigskip
Here, we determine the exact value of the size Ramsey number of two stars of the same sizes versus an arbitrary star forest. 

\begin{theorem}\label{thm:s=2}
	For positive integers $ n $ and $ m_1\geq m_2\geq\cdots\geq m_t\geq2 $, let $ F_1=2K_{1,n} $ and $ F_2=\sqcup_{i=1}^t K_{1,m_i} $. Then  $\hat{r}(F_1,F_2)=n+m_1-1+\sum_{i=1}^{t}(n+m_i-1).$
\end{theorem}

\begin{proof}
	It is sufficient to show that if $G$ is a graph such that $G\rightarrow(F_1,F_2)$, then $e(G)\geq n+m_1-1+\sum_{i=1}^{t}(m_i+n-1)$.
	
	Let $ G $ be a graph such that $ G\rightarrow(F_1,F_2) $. If $ \Delta(G)\leq n+m_1-3 $, then by Lemma \ref{decompose}, there exists a $(K_{1,n},K_{1,m_1})$-free coloring of $G$, contradicting  $ G\rightarrow(F_1,F_2) $. Now, let $ \Delta(G)\geq m_1+n-1 $. Thus, there is a vertex $ v\in V(G) $ such that $ \deg(v)\geq m_1+n-1 $.
	
	\medskip
	With the same argument as in the proof of Theorem \ref{starsss}, $G-v\rightarrow(K_{1,n}, F_2)$. Thus,  by Theorem~\ref{star-forest}, $e(G-v)\geq \sum_{i=1}^{t}(m_i+n-1) $ and so, $ e(G)\geq \deg(v)+e(G-v)\geq n+m_1-1+\sum_{i=1}^{t}(m_i+n-1).$

	Hence, we may assume that $ \Delta(G)=n+m_1-2 $ and $ \deg(v)=n+m_1-2 $.
	If $ e(G-v)> \sum_{i=1}^{t}(m_i+n-1) $, then $e(G)> n+m_1-2+\sum_{i=1}^{t}(m_i+n-1)$ and we are done. Thus, we may assume that $ e(G-v)= \sum_{i=1}^{t}(m_i+n-1) $. As, $ G-v\rightarrow(K_{1,n}, F_2) $, by Theorem~\ref{star-forest}, $ G-v=\sqcup_{i=1}^tG_i $ such that for every $i$, $ G_i=K_{1,n+m_i-1} $ or $ G_i=K_3 $.\\
	If $ G_j=K_{1,n+m_1-1} $ for some $ j $, then $ \Delta(G)\geq n+m_1-1 $ which contradicts $\Delta(G)=n+m_1-2 $.
	Therefore,  $ G-v=tK_3 $ and for all $i$, $1\leq i\leq t$, we have $ n=m_i=2 $. Since $ \Delta(G)=2 $, $ v $ has no neighbor in $ G-v $ and so $ G=K_{1,2}\cup tK_3 $. Now, we color an arbitrary edge of each component of $ G $ by red and the rest by blue. Then, we exchange the color of a blue edge in exactly one of the triangles by red. Clearly, there is only one red copy of $ K_{1,2} $ and $ t-1 $ blue copies of $ K_{1,2} $ which contradicts $ F\rightarrow(2K_{1,n},\cup_{i=1}^tK_{1,m_i}) $.
\end{proof}

\bigskip
Now, in the sequel, we determine the size Ramsey number of star forests containing stars of odd sizes. 

%
%
\begin{theorem}\label{odd-starforest}
	Let  $n_{1}\geq n_{2}\geq \cdots\geq n_s$ and $m_1\geq m_2\geq \cdots\geq m_t$ be odd positive integers and $F_1=\sqcup_{i=1}^{s}K_{1,n_i}$ and $F_2=\sqcup_{j=1}^{t}K_{1,m_j}$. Then, $\hat{r}(F_1,F_2)= \sum_{k=2}^{s+t}l_k,$ where, $l_k=\max\{n_i+m_j-1: ~ i+j=k\}.$
\end{theorem}

\begin{proof}
Suppose that $G$ is a graph such that $G\rightarrow (F_1,F_2)$. We are going to prove that $e(G)\geq \sum_{k=2}^{s+t}l_k$. First, note that by Lemma~\ref{decompose}, $\Delta(G)\geq m_1+n_1-1$.

	Now, let $v_1$ be a vertex of degree at least $l_2$ in $G$. Then, one may see that $G-v_1 \rightarrow (\sqcup_{i=2}^s K_{1,n_i}, F_2)$ and $G-v_1 \rightarrow (F_1, \sqcup_{j=2}^t K_{1,m_j})$. Thus, again by Lemma~\ref{decompose}, we have $\Delta(G-v_1) \geq \max\{n_2+m_1-1, n_1+m_2-1\}=l_3$. Now, choose a vertex $v_2$ of degree at least $l_3$ in $G-v_1$. By continuing this process, one we may find vertices $v_1,v_2,\ldots, v_{s+t-1}$ such that for all $k\in\{1,\ldots s+t-1\}$, the degree of $v_k$ in $G\setminus\{v_1,\ldots,v_{k-1}\}$ is at least $l_{k+1}$. This proves that $e(G)\geq \sum_{k=2}^{s+t} l_k$. This completes the proof.
\end{proof}

\bigskip
As the last result of this paper, we prove the following theorem which determines the size Ramsey number of arbitrary star forests under a certain condition.
\begin{theorem}\label{odd-stars-forest}
	Let $s$, $n$ and $m_1\geq m_2\geq \cdots\geq m_t\geq 2$ be positive integers. If $n$ and $m_1$ are both odd, then $$\hat{r}(sK_{1,n}, \sqcup_{j=1}^{t}K_{1, m_j})=(s-1)(n+m_1-1)+\sum_{j=1}^{t} (n+m_j-1).$$
	Moreover, if $n$ and $m_1$ are both odd, $G\rightarrow(sK_{1,n}, \sqcup_{j=1}^{t}K_{1, m_j})$ and $e(G)=(s-1)(n+m_1-1)+\sum_{j=1}^{t} (n+m_j-1)$, then $G=(s-1)K_{1,n+m_1-1}\sqcup_{j=1}^{t}K_{1,n+m_j-1}$.
\end{theorem}
\begin{proof}
	We use induction on $s$ to prove the theorem. The base case $s=1$, follows from Theorem \ref{star-forest}. Let $s\geq 2$ and $G$ be a graph such that $G\rightarrow(sK_{1,n}, \sqcup_{j=1}^{t}K_{1, m_j})$.

	\medskip
	First, by Lemma~\ref{decompose}, we have $\Delta(G)\geq n+m_1-1$. 
	
	\medskip
	Now, let $v$ be a vertex of maximum degree in $G$, i.e. $\deg(v)\geq n+m_1-1$. By a similar argument, used in the proof of Theorem \ref{starsss}, $(G-v) \rightarrow((s-1)K_{1,n}, \sqcup_{j=1}^{t}K_{1, m_j})$. By the induction hypothesis, $e(G-v)\geq (s-2)(n+m_1-1)+\sum_{j=1}^{t} (n+m_j-1)$ and so $$e(G)=\deg(v)+ e(G-v)\geq (s-1)(n+m_1-1)+\sum_{j=1}^{t} (n+m_j-1).$$
This observation shows that $$\hat{r}(sK_{1,n}, \sqcup_{j=1}^{t}K_{1, m_j})=(s-1)(n+m_1-1)+\sum_{j=1}^{t} (n+m_j-1).$$

Now, let $G\rightarrow(sK_{1,n}, \sqcup_{j=1}^{t}K_{1, m_j})$ and $e(G)=(s-1)(n+m_1-1)+\sum_{j=1}^{t} (n+m_j-1)$. Also, suppose that $m_1=\cdots=m_\ell$ and $m_{\ell+1}<m_\ell$, for some $\ell\geq 1$. By the above argument, we have $\deg(v)=n+m_1-1$ and  $e(G-v)=(s-2)(n+m_1-1)+\sum_{j=1}^{t} (n+m_j-1)$. By the induction hypothesis, $G-v=(s-2)K_{1,n+m_1-1}\sqcup_{j=1}^{t}K_{1,n+m_j-1}$. We claim that $N(v)\cap V(G-v)=\emptyset$.

On the contrary, let $u\in N(v)\cap V(G-v)$.  Set $A=(s-1)K_{1,n+m_1-1}$ and $B=\sqcup_{j=2}^{t}K_{1,n+m_j-1}$. First, suppose that $u\in N(v)\cap S$, for some component $S$ of $A$. Note that the maximum degree condition on $G$ implies that $u$ is not the root vertex of $S$.  In this case, color $n$ edges of each component in $A$ and also $n-1$ edges of each component in $B$ by red and the rest by blue. Also, color $n$ edges incident with $v$ by red and the rest by blue such that edges incident with $u$ are red.  It is obvious that the red subgraph contains at most $s-1$ disjoint copies of $K_{1,n}$ and the blue subgraph contains at most $\ell-1$ disjoint copies of $K_{1,m_1}$. Hence, we have a $(sK_{1,n}, \sqcup_{j=1}^{t}K_{1, m_j})$-free coloring of $G$, a contradiction.
 
 \medskip
Now, let $u\in N(v)\cap S$ for some component $S$ of $B$.  In this case, color all edges in $A$ and also $n-1$ edges of each component in $B$ by red and the rest by blue. Also, color $n-1$ edges incident with $v$ by red and the rest by blue such that $S\cup N(v)$ contains a blue copy of $P_4$.  Clearly, there are at most $(s-1)$ disjoint red copies of $K_{1,n}$. Also, the number of blue edges is exactly $ \sum_{j=1}^{t} m_j $.  Existence of a blue $ P_4 $ in the blue subgraph implies that $G$ contains no blue copy of $\sqcup_{j=1}^{t}K_{1, m_j}$ and so we have a $(sK_{1,n}, \sqcup_{j=1}^{t}K_{1, m_j})$-free coloring of $G$, a contradiction. This contradiction shows that $N(v)\cap V(G-v)=\emptyset$ and so $G=(s-1)K_{1,n+m_1-1}\sqcup_{j=1}^{t}K_{1,n+m_j-1}$. This completes the proof.
\end{proof}

\section{Concluding Remarks}
We close the paper with some supplementary remarks. First, one may think of a generalization of the results of the paper to the multicolor size Ramsey numbers. Let $p_1, p_2,\ldots, p_t$ and $n_{i_1}\geq n_{i_2}\geq \cdots\geq n_{i_{p_i}}$, $(1\leq i\leq q)$, be positive integers. Also, let $F_1, F_2, \ldots, F_q$ be star forests such that for $i=1, 2, \ldots, q$,  $F_i=\sqcup_{j=1}^{p_i}K_{1,n_{i_j}}$. Set $p=\sum_{i=1}^{q}p_i$ and for $k=q, q+1, \ldots, p$, suppose that
$$l_k=\max\{(n_{1_{j_1}}-1)+(n_{2_{j_2}}-1)+ \cdots+(n_{t_{j_t}}-1)+1: ~ j_1+ j_2+ \cdots+ j_t=k\}.$$

\medskip
Using the simple fact that $K_{1,(m_1-1)+(m_2-1)+\cdots+(m_q-1)+1}\rightarrow(K_{1,m_1},K_{1,m_2},\ldots,K_{1,m_q})$, we deduce that if $F=\sqcup_{k=q}^{p}K_{1,l_k}$, then $F\rightarrow(F_1,F_2,\ldots,F_q)$. Thus,
\begin{equation}\label{eq0}
	\hat{r}(F_1,F_2,\ldots,F_q)\leq\sum_{k=q}^{p}l_k.
\end{equation}

We believe that equality holds in \eqref{eq0} which is stated as the following conjecture and can be considered as an extension of Conjecture \ref{conj} to the multicolor case.

\begin{conjecture}\label{ourconj}
	Let $p_1,p_2,\ldots,p_q$ and $n_{i_1}\geq n_{i_2}\geq \cdots\geq n_{i_{p_i}}$, $(1\leq i\leq q)$, be positive integers  and $p=\sum_{i=1}^{q}p_i$. If $F_1,F_2,\ldots,F_q$ are star forests such that $F_i=\sqcup_{j=1}^{p_i}K_{1,n_{i_j}}$, then
	$$\hat{r}(F_1,F_2,\ldots,F_q)= \sum_{k=q}^{p}l_k,$$
	where, $l_k=\max\{(n_{1_{j_1}}-1)+(n_{2_{j_2}}-1)+ \cdots+(n_{q_{j_q}}-1)+1: ~ j_1+ j_2+ \cdots+ j_q=k\}$.
\end{conjecture}

Using the same methods as in the proofs, we can generalize the results of Theorems \ref{starsss}, \ref{star-forest}, \ref{odd-starforest}, \ref{odd-stars-forest}. In fact, we have the exact values of the following parameters.

\begin{itemize}
	\item[$\bullet$] $\hat{r}(s_1K_2, s_2K_2, \ldots, s_qK_2, \sqcup_{j=1}^{t}K_{1, m_j}),$
	
	\item[$\bullet$] $\hat{r}(K_{1,n_1}, K_{1,n_2},\ldots,K_{1,n_q}, \sqcup_{j=1}^{t}K_{1, m_j}),$
	
	\item[$\bullet$] $\hat{r}(F_1,F_2,\ldots,F_q)$, when each component of $F_i$'s is odd,
	
	\item[$\bullet$] $\hat{r}(s_1K_{1,n_1}, s_2K_{1,n_2},\ldots, s_qK_{1,n_q}, \sqcup_{j=1}^{t}K_{1, m_j})$, when for  each $i$, $n_i$ and $m_1$ are odd.
\end{itemize}


\end{document}